\documentclass[letterpaper, 10 pt, journal, twoside]{ieeetran} 

\IEEEoverridecommandlockouts            

\usepackage{titlesec}
\usepackage{microtype}
\usepackage{hyperref}
\usepackage{url}

\usepackage[noadjust]{cite}

\usepackage{microtype}
\usepackage{graphicx}
\usepackage{caption}
\usepackage{subcaption}
\usepackage{wrapfig}
\usepackage{booktabs} 

\usepackage{algorithm}
\usepackage{algorithmic}
\usepackage{lipsum}
\usepackage{xcolor}
\usepackage{amssymb,amsmath}
\usepackage{amsthm}

\newtheorem{remark}{Remark}
\newtheorem{definition}{Definition}
\newtheorem{lemma}{Lemma}
\newtheorem{theorem}{Theorem}

\graphicspath{{./Images/}}
\title{\LARGE \bf
Heterogeneous Explore-Exploit Strategies on Multi-Star Networks
}

\author{Udari Madhushani, \IEEEmembership{Student Member, IEEE}, Naomi Ehrich Leonard, \IEEEmembership{Fellow, IEEE}
\thanks{This research has been supported in part by ONR grants N00014-18-1-2873 and N00014-19-1-2556 and ARO grant W911NF-18-1-0325. \textit{(Corresponding author: Udari Madhushani.)}}
\thanks{The authors are with the Department of Mechanical and
Aerospace Engineering, Princeton University, Princeton, NJ 08544
USA (e-mail: udarim@princeton.edu; naomi@princeton.edu).}}

\begin{document}

\maketitle
\thispagestyle{empty}
\pagestyle{empty}


\begin{abstract}

We investigate the benefits of heterogeneity in multi-agent explore-exploit decision making where the goal of the agents is to maximize cumulative group reward. To do so we study a class of distributed stochastic bandit problems in which agents communicate over a multi-star network  and make sequential choices among options in the same uncertain environment.  
Typically, in multi-agent bandit problems, agents use homogeneous decision-making strategies. However, group performance can be improved by incorporating heterogeneity into the choices agents make, especially when the network graph is irregular, i.e. when agents have different numbers of neighbors.  We design and analyze new heterogeneous explore-exploit strategies, using the multi-star as the model irregular network graph.  The key idea is to enable center agents to do more exploring than they would do using the homogeneous strategy, as a means of providing more useful data to the peripheral agents.
In the case all agents  broadcast their reward values and choices to their neighbors with the same  probability, we provide theoretical guarantees that group performance improves under the proposed heterogeneous  strategies as compared to under homogeneous strategies. We use numerical simulations to illustrate our results and to validate our theoretical bounds.  

\end{abstract}

\begin{IEEEkeywords}
Bandit algorithms, distributed learning, heterogeneous strategies
\end{IEEEkeywords}


\section{Introduction}\label{sec:intro}

\IEEEPARstart{T}{he} influence of agent heterogeneity on cooperation in social learning has been a recent focus of research in many fields, including ecology, sociology, and decision theory \cite{li2019influence}. Studies on evolutionary human behavior provide evidence that individual differences can be leveraged to enhance collective prosperity \cite{amaral2016evolutionary}. Motivated by applications such as  social foraging and multi-robot coordination tasks, we study and design cooperative strategies for a group of agents making sequential explore-exploit decisions in an uncertain environment. The strategies we design incorporate agent heterogeneity to optimize the performance of the group through collective learning.

Consider a group of agents, each making a sequence of choices among options in an uncertain environment in order to maximize collective payoff. At each time step in the sequence, each agent chooses an option depending on the knowledge it has acquired about the environment up to that time step. Maximizing payoff  necessitates striking a balance between making choices that yield high immediate payoff, i.e., exploiting, and making choices that yield high information content and possibly high future payoffs, i.e., exploring. When an agent fails to acquire sufficient information about the environment to make optimal decisions, it must sacrifice exploitation potential in order to explore. However, in the group setting, agents can recover exploitation potential by gaining information through cooperation i.e., through collective learning. 

Sequential decision making in uncertain environments that requires trading off exploitation and exploration is modeled mathematically by the bandit framework \cite{robbins1952some}. In the multi-armed bandit (MAB) problem, an agent is repeatedly faced with the task of choosing an option from a given set of options. At each time step the agent receives a stochastic reward drawn from a fixed probability distribution associated with the chosen option. The agent's goal is to maximize the  cumulative reward by the end of the decision-making process. This requires choosing frequently enough the optimal option i.e., the option with highest expected reward. In order to meet this requirement, the agent must simultaneously choose options that are known to provide high rewards (exploit) and choose lesser known options (explore) that might potentially provide even higher rewards \cite{lai1985asymptotically, auer2002finite}. 

Maximizing   cumulative reward is equivalent to minimizing   cumulative regret, defined as the loss incurred by an agent  choosing a sub-optimal option instead of the optimal option. Since the probability distribution associated with each option is fixed, cumulative regret can be minimized by reducing the number of times sub-optimal options are chosen. Performance of the proposed algorithms for this problem is measured using expected cumulative regret. The paper \cite{lai1985asymptotically} establishes that any efficient policy chooses suboptimal options asymptotically logarithmically in time. The paper \cite{auer2002finite} proposes an Upper Confidence Bound (UCB) based sampling rule that achieves a logarithmic expected cumulative regret uniformly in time.

The papers \cite{landgren2016distributedCDC,landgren2020distributed,martinez2019decentralized,szorenyi2013gossip,madhushani2020dynamic,madhushani2020distributed,chakraborty2017coordinated,madhushani2019heterogeneous,landgren2018social,wang2020optimal,kolla2018collaborative} extend to the multi-agent setting and capture different aspects of collective learning. 
In \cite{landgren2016distributedCDC,landgren2020distributed,martinez2019decentralized,szorenyi2013gossip}, agents share their estimates of the expected reward of options with neighbors according to fixed communication structures.  The papers \cite{landgren2016distributedCDC,landgren2020distributed}  use a running consensus algorithm to update estimates and provide graph-structure-dependent performance measures that predict the relative performance of agents and networks. The paper \cite{landgren2020distributed} also addresses the case of a constrained reward model in which agents that choose the same option at the same time step receive no reward.  The paper \cite{martinez2019decentralized} 
proposes an accelerated consensus procedure assuming agents know the spectral gap of the communication graph and designs a decentralized UCB algorithm based on delayed rewards. The paper \cite{szorenyi2013gossip} considers a P2P communication where an agent is only allowed to communicate with two other agents at each time step.

The papers \cite{madhushani2020dynamic,madhushani2020distributed,chakraborty2017coordinated,madhushani2019heterogeneous,landgren2018social,wang2020optimal,kolla2018collaborative} consider the case in which agents  share reward values and choices with neighbors. In \cite{madhushani2020dynamic,madhushani2020distributed,chakraborty2017coordinated},  agents use stochastic communication structures that depend on the decision-making process. In \cite{madhushani2020dynamic}, each agent observes rewards and actions of its neighbors when it is exploring. In \cite{madhushani2020distributed}, each agent instead broadcasts its rewards and actions to its neighbors when it is exploring. In \cite{chakraborty2017coordinated}, at each time step, agents decide either to sample an option or to broadcast the last obtained reward to the entire group. 

The setup in our earlier paper \cite{madhushani2019heterogeneous} is closest to that in the present paper: agents observe reward values and actions of their neighbors defined by a network graph that changes in time according to probabilistic edge weights. An underlying fixed network graph is given, and each agent $k$ observes its neighbors with probability $p_k$.  The communication structure is independent of the decision-making process. 

The papers \cite{landgren2016distributedCDC,landgren2020distributed,martinez2019decentralized,szorenyi2013gossip,madhushani2020dynamic,madhushani2020distributed,chakraborty2017coordinated,madhushani2019heterogeneous} consider homogeneous protocols, whereas the papers \cite{landgren2018social,wang2020optimal,kolla2018collaborative} consider protocols where some agents (followers) copy actions of others (leaders). In \cite{landgren2018social}, followers observe  rewards and choices of their neighbors. In \cite{wang2020optimal}
 one leader explores and estimates the mean reward of options, while all other agents choose the option with highest estimated mean per the leader. The paper \cite{kolla2018collaborative} proposes the FYL algorithm, which uses a deterministic communication protocol and exploits degree heterogeneity of the communication network graph. FYL outperforms our algorithm when $p_k = p=1$; however, our algorithm provides a method to exploit agent heterogeneity when agents share information with probability $0<p<1$.
 
When communication among agents is defined by an irregular network graph, e.g., some agents serve as information hubs, group performance can be improved by using heterogeneous explore-exploit strategies. 
To understand this, consider an environment with unconstrained resources.  Then, agents can only influence the decisions of one another through the information they share, and the structure of interactions that defines neighbors, i.e., who is sharing information with whom, strongly affects the quality and quantity of information received by each individual. 

We consider the case that all agents broadcast their instantaneous rewards and actions to their neighbors with probability $p$. This communication protocol is motivated by real-world applications in which estimates of mean rewards or the sum of collected rewards, which rely on the history of choices and rewards, are deliberately not disclosed to protect privacy \cite{feraud2018decentralized}. 
For example, in user targeted recommender systems \cite{warlop2018fighting} (or clinical trials \cite{tossou2016algorithms}), sharing user (patient) history of  choices can reveal sensitive information about users (patients). Even when an agent is broadcasting only its current rewards and actions to neighbors, an adversarial agent can listen to the broadcasts and access the history of choices made by the agent. To reduce such privacy leakage we consider  agents that broadcast instantaneous rewards and actions probabilistically.  Further, if communication failures are possible, then having agents broadcast only current rewards and actions avoids problems associated with agents losing track of what information has and has not been received by neighbors.  In this context, $1-p$ represents the probability of communication failure.

In irregular and centralized networks like the multi-star, center agents have more neighbors and thus  receive more information than peripheral agents. This leads to an imbalanced exploitation potential across the group \cite{sueur2012social,li2019influence}, and group performance degrades with increasing number of peripheral agents. We investigate improving group performance by leveraging heterogeneity in the exploitation potential of agents. To do so we propose  heterogeneous explore-exploit strategies that require center agents to explore more and thus increase the exploitation potential of peripheral agents.

The multi-star network models recommender systems, where there are many small servers, assigned to different regions, that each make sequential recommendations based on user feedback and communicate only with a large central server.
Performance can be improved by using the central server to suggest more exploratory recommendations which allows the system to gather more information about user preferences. Probabilistic communication accounts for random communication failures between servers.

The paper is organized as follows. In Section \ref{sec:problem} we provide the problem formulation and notation. Section \ref{sec:algo} presents the proposed algorithm and intuition. We analyze  performance of the proposed algorithm in Section \ref{sec:performance} and provide improved theoretical bounds for the expected cumulative group regret.  In Section \ref{sec:simulation} we show numerical simulations to illustrate and validate the theoretical results. We conclude in Section \ref{sec:conl}. 


\section{Problem Formulation}\label{sec:problem}
In this section we present the problem formulation and relevant mathematical notations. Consider a group of $K$ agents, each faced with the same $N$-armed bandit problem for $T$ time steps. At each time step $t \in \{1, \ldots, T\}$, each agent chooses an option and receives a stochastic reward associated with the chosen option. Let $X_i$ be a sub-Gaussian random variable that denotes the {\em reward associated with option} $i\in \{1,\ldots,N\}.$ Sub-Gaussian rewards include widely used distributions such as Bernoulli, Gaussian, and bounded rewards. Define $\mu_i=\mathbb{E}(X_i)$ and $\sigma^2_i$ as the {\em expected reward} and {\em variance proxy} associated with option $i$, respectively. Let $i^*=\arg_{i} \max \{\mu_1,\ldots, \mu_N\}$ be the {\em optimal option} with highest expected reward. Define $\Delta_i=\mu_{i^*}-\mu_i$ as the {\em expected reward gap} between option $i^*$ and option $i.$

Let $G(\mathcal{V},\mathcal{E})$ be a fixed undirected network graph that defines the structure of the interactions between agents. This captures the inherent hard communication constraints of the system. Here $\mathcal{V}$ is a set of $K$ vertices such that each vertex corresponds to an agent. Each edge $e(k,j)\in \mathcal{E}$ in the graph denotes that agent $k$ and agent $j$ are {\em neighbors}. At each time step, each agent broadcasts its reward value and action to its neighbors with {\em broadcasting probability} $p$. Let $\mathbb{I}^t_{\{k,j\}}$ be the indicator random variable that takes value 1 if agent $k$ receives information from agent $j$ at time $t$ and 0 otherwise. Then, for every time $t$,  $\mathbb{E}(\mathbb{I}^t_{\{k,j\}})=p,\forall k,j$ such that $e(k,j)\in \mathcal{E}$, and $\mathbb{E}(\mathbb{I}^t_{\{k,j\}})=0$ otherwise. We define $\mathbb{I}_{\{k,k\}}^t=1,\forall k,t.$ 

Let $d_k$ be the {\em degree} (number of neighbors) of agent $k$ and $d_{avg}=\frac{1}{K}\sum_{k=1}^K d_k$ be the {\em average degree of the network}. Let $d^{avg}_k$ be the {\em average degree of neighbors of agent $k$}: $
d^{avg}_k=\frac{1}{d_k}\sum^K_{e(k,j)\in\mathcal{E}} d_j.
$

We focus on {\em multi-star} graphs defined as follows. Let there be $m$ center agents and $K-m$ peripheral agents.  Without loss of generality let each agent $k$, $k \leq m$, be a {\em center agent}.  All center agents are neighbors of one another, i.e., $e(k,j)\in \mathcal{E}, \forall k,j \leq m$, and a center agent's degree $d_k$ is at least $m-1$. Each {\em peripheral agent} $k$, $k > m$, has exactly one neighbor ($d_k=1$), and the neighbor is a center agent. To reduce complexity, we assume the graph is symmetric, which implies that all  center agents have the same number of neighbors. Thus $K-m$ is an integer multiple of $m$. If $K > 2$ and $m<K$, the multi-star graph is {\em irregular}, i.e., the degree of center agents differs from the degree of peripheral agents. Let $d_{cen}$ be the degree of each center agent. Then, $d_{cen}=\frac{K-m}{m}+m-1.$ When $m=1$ the graph is a star, the most irregular multi-star graph. When $m=K$, there are no peripheral agents and the graph is all-to-all and thus regular.

Let $\varphi_t^k$ be a random variable that denotes the {\em option chosen by agent $k\in \{1,\ldots,K\}$ at time $t\in\{1,\ldots,T\}$}. Let $\mathbb{I}_{\{\varphi_t^k=i\}}$ be an indicator random variable that takes value 1 if agent $k$ chooses option $i$ at time $t$ and 0 otherwise. Let $n_i^k(t)$  be the {\em total number of times agent $k$ chooses option $i$ until time $t$} and let $N_i^k(t)$ be the {\em total number of times agent $k$ observes option $i$ until time $t$}. The total number of observations is the sum of the number of samples taken from option $i$ by agent $k$ and the number of broadcasts on option $i$ by its neighbors:
\begin{align}
n_i^k(t)=\sum_{\tau=1}^t\mathbb{I}_{\{\varphi_{\tau}^k=i\}},\:\:\:N_i^k(t)=\sum_{\tau=1}^t\sum_{ j=1}^K\mathbb{I}_{\{\varphi_{\tau}^j=i\}}\mathbb{I}^{\tau}_{\{k,j\}}.
\label{eq:nandN}
\end{align}
Let $\widehat{\mu}_i^k(t)$ denote the estimate of expected reward of agent $k$ for option $i$ at time $t.$ Then, $
\widehat{\mu}^k_i(t)=\frac{S_i^k(t)}{N^k_i(t)}$, where $S_i^k(t)=\sum_{\tau=1}^t\sum_{j=1}^K X_i\mathbb{I}_{\{\varphi^j_{\tau}=i\}}\mathbb{I}^{\tau}_{\{k,j\}}.$ 

{\em Expected regret} is defined as the expected loss suffered by agents by sampling sub-optimal options. Let $R(t)$ be the {\em cumulative group regret} at time $t.$ Then {\em expected cumulative group regret} can be computed as
\begin{align}
\mathbb{E}\left(R(t)\right)=\sum_{i=1}^N\sum_{k=1}^K\Delta_i\mathbb{E}\left(n_i^k(t)\right)\label{eq:regret}.
\end{align}


\section{Algorithm}\label{sec:algo}
To realize the goal of maximizing cumulative group reward, agents should minimize the number of times they sample sub-optimal options. Each agent employs an agent-based strategy that captures the trade-off between exploring and exploiting by constructing an objective function that strikes a balance between the estimation of the expected reward and the uncertainty associated with the estimate \cite{auer2002finite}. 

Since center agents have more neighbors they are more likely to obtain a high number of observations. This reduces the uncertainty associated with their estimate of the expected reward of options. Thus,  identifying the optimal option requires less exploring, which increases their exploitation potential. Since peripheral agents only have one neighbor they are more likely to obtain a low number of observations. Thus, identifying the optimal option requires more exploring, which decreases their exploitation potential. Further, since center agents do less exploring, the usefulness of the information they broadcast is reduced, also decreasing the peripheral agents' exploitation potential. Accordingly, homogeneous sampling rules in irregular, multi-star networks lead to imbalanced exploitation potential across the group and thus degraded group performance.

To improve group performance, we propose heterogeneous explore-exploit strategies that regulate  exploitation potential across the network. When center agents are more exploratory their performance degrades, but the usefulness of the information they broadcast increases and so the performance of peripheral agents improves. When there are more peripheral agents than center agents, and broadcasting probability $p$ is sufficiently high, the performance improvement obtained by peripheral agents outweighs the performance degradation incurred by center agents, and group performance increases.  If $p$ is too small, for example, when broadcasting is costly or risky, center agents do not broadcast enough information to benefit peripheral agents. Thus it doesn't pay for center agents to increase their exploration. Indeed, when $p=0$ all  agents have the same exploitation potential. 

Using this intuition, we propose the following heterogeneous sampling rules. 
Assume that variance proxy  $\sigma^2_i$ for each option $i$ is known to all agents. 
\begin{definition} {\bf{(Heterogeneous Sampling Rules)}}\label{def:samplerule}
The sampling rule $\{\varphi^k_t\}_1^{T}$ of agent $k$ at time $t \in \{1, \ldots, T\}$ is
    \begin{align*}
    \mathbb{I}_{\{\varphi^k_{t+1}=i\}}=\left\{
    \begin{array}{cl} 1 &, \:\:\:i=\arg \max \{Q^k_1(t),\cdots,Q^k_N(t)\}\\
      0 &, \:\:\: {\mathrm{o.w.}}\end{array}\right.
    \end{align*}
    with     
    \begin{align}
    Q^k_i(t)&= \widehat{\mu}^k_{i}(t)+C^k_i(t)\label{eq:UCBQ}\\
    C^k_i(t)&=\sigma_{i}\sqrt{\frac{2(1+\alpha_k)(\xi+1)\log t}{N^k_{i}(t)}}\label{eq:Uncertainity}
    \end{align}
    where $\xi>1$ and 
    \begin{align}
\alpha_k=\left\{
    \begin{array}{cl} 
     \frac{ p^{1-p}(d_k-d^{avg}_k)}{d_k} &, \:\:\: k\leq m \\
     0 &, \:\:\:k>m.
     \end{array}\right. 
     \label{eq:alpha}
    \end{align}
\end{definition}

$C_i^k(t)$ in \eqref{eq:Uncertainity} represents {\em agent $k$'s uncertainty in its estimated mean of option $i$}, and Definition~\ref{def:samplerule} implies that for any agent $k$,  when $C_i^k(t)$ is high, agent $k$ will more likely explore. By \eqref{eq:Uncertainity}, $C_i^k(t)$ can be high when $N_i^k$, the number of agent $k$'s observations of option $i$, is low, i.e., when option $i$ is under-sampled.  $C_i^k(t)$ can also be high when {\em agent  $k$'s exploration bias} $\alpha_k>0$ is high. 

By \eqref{eq:alpha}, $\alpha_k \neq 0$ only for center agents. Since peripheral agents have one center agent neighbor, $d_k^{avg} \leq d_k$ and thus $\alpha_k \geq 0$ for every center agent $k \leq m$. In fact, $\alpha_k\geq 0$ is designed to grow with increasing irregularity: in the regular case (all-to-all) when $m=K$, $\alpha_k = 0$, and in the most irregular case (star) when $m=1$, $d_1 = K-1$ and $d_1^{avg} = 1$ so $(d_1-d_1^{avg})/d_1 = (K-2)/(K-1)$. Further, $\alpha_k$ grows with $p$ according to the factor $p^{1-p}$, which grows rapidly for intermediate values of $p$ and is large (i.e., saturates to 1) only when center agents are broadcasting their reward values and actions with sufficiently high probability $p$.

\begin{definition}\label{def:homrule}
To get the {\em corresponding homogeneous sampling rules} let $\alpha_k=0,\forall k$, in Definition~\ref{def:samplerule}. Heterogeneous and homogeneous rules for peripheral agents are the same.
\end{definition}

By design, the heterogeneous  rules of Definition \ref{def:samplerule} drive center agents to explore more than  the corresponding homogeneous  rules and only when it benefits group performance.

\section{Performance Analysis}\label{sec:performance}
In this section we analyze the performance of the heterogeneous sampling rules of Definition~\ref{def:samplerule}.  Using an  approach similar to \cite{auer2002finite} with a few key modifications, we upper bound the expected cumulative group regret $\mathbb{E}(R(T))$. We  show that the bound is lower than the upper bound in the case of the corresponding homogeneous sampling rules, and so we can conclude that the designed heterogeneous strategies provide better group performance than the homogeneous strategies.

By (\ref{eq:regret}), we  upper bound  $\mathbb{E}(R(T))$ if we   upper bound  $\sum_{k=1}^{K} \mathbb{E}(n_i^k(T))$, where  $n_i^k(T)$ is the number of times   agent $k$ samples sub-optimal option $i$ until time $T$. By Definition~\ref{def:samplerule}, agent $k$ chooses sub-optimal option $i$ at time  $t$ if $Q_i^k(t)\geq Q_{i^*}^k(t)$.  
Then, $
n_i^k(t)=\sum_{\tau=1}^t\mathbb{I}_{\{\varphi_\tau^k=i\}}\leq \sum_{\tau=1}^t\mathbb{I}_{\{Q_i^k(\tau)\geq Q_{i^*}^k(\tau)\}}.$
For each option $i$ and agent $k$ let $\{\eta_i^k(t)\}_1^T$ be a sequence of nonnegative nondecreasing functions. Then,
\begin{align}
\sum_{k=1}^K\mathbb{E}&\left(n_i^k(T)\right) 
\leq\sum_{k=1}^K\sum_{t=1}^T\mathbb{E}\left(\mathbb{I}_{\{\varphi_t^k=i\}},N_i^k(t)\leq\eta_i^k(t)\right)\nonumber\\
+&\sum_{k=1}^K\sum_{t=1}^T\mathbb{P}\left(Q_i^k(t)\geq Q_{i^*}^k(t),N_i^k(t)>\eta_i^k(t)\right). \label{neq:step1regret}
\end{align}

It remains to upper bound the right hand side of  (\ref{neq:step1regret}) and we do so in two steps. First, we  upper bound the second summation term of  (\ref{neq:step1regret}) as follows. From  (\ref{eq:UCBQ}) we have
\begin{align}
&\left\{Q_{i}^{k}(t)\geq Q_{i^{*}}^{k}(t)\right\}\subseteq \left\{\mu_{i^{*}}<\mu_{i}+2C_{i}^{k}(t)\right\}\nonumber\\
&\cup\left\{\widehat{\mu}_{i^{*}}^{k}(t)\leq \mu_{i^{*}}-C_{i^{*}}^{k}(t)\right\}\cup \left\{\widehat{\mu}_{i}^{k}(t)\geq \mu_{i}+C_{i}^{k}(t)\right\}\label{eq:seteq}.
\end{align}
For all $k$ let
\begin{align}
\eta_i^k(t)=(1+\alpha_k)\eta_i(t), \;\;\; \eta_i(t)=\frac{8\sigma_i^2(\xi+1)\log t}{\Delta_i^2}.
\label{eq:eta}
\end{align}
Then, by (\ref{eq:Uncertainity}),  $\{\mu_{i^{*}}<\mu_{i}+2C_{i}^{k}(t)\}\cap\{N_i^k(t)> \eta_i^k(t)\}=\emptyset$ where $\emptyset$ is the empty set.
Using (\ref{eq:seteq}) we obtain
\begin{align}
&\mathbb{P}\left(Q_i^k(t)\geq Q_{i^*}^k(t),N_i^k(t)>\eta_i^k(t)\right)\leq\nonumber\\
 &\mathbb{P}\left(\widehat{\mu}_{i^{*}}^{k}(t)
\leq \mu_{i^{*}}-C_{i^{*}}^{k}(t)\right)
+\mathbb{P}\left(\widehat{\mu}_{i}^{k}(t)
\geq \mu_{i}+C_{i}^{k}(t)\right).\label{eq:twotail}
\end{align}
To upper bound the right hand side of \eqref{eq:twotail} we use the tail probability bound provided in the following lemma.
\begin{lemma} \label{lem:tailProb}
For any $\xi>1,$ some $\zeta>1$ and for $\sigma_i>0$ in the uncertainty $C_i^k(t)$ given by \eqref{eq:Uncertainity}, we get 
\begin{align*}
    \mathbb{P}\left(\big |\widehat{\mu}_i^k(t)-{\mu}_i\big |>C_i^k(t)\right) \leq \frac{1}{\log \zeta}\frac{\log \left((1+d_k)t\right)}{t^{(\xi+1)(1+\alpha_k)}}.
\end{align*}
\end{lemma}
\begin{proof}
From Theorem 1 in the paper \cite{madhushani2019heterogeneous} we have
for some $\zeta>1$ and for $\sigma_i>0$ there exists a $\vartheta_k>0$ such that
\begin{align*}
    \mathbb{P}\left(\widehat{\mu}_i^k(T)-{\mu}_i>\sqrt{\frac{\vartheta_k}{N_{i}^{k}(T)}}\right)\leq \frac{\nu\log ((d_k+1)T)}{\exp(2\kappa \vartheta_k)}
\end{align*}
where, $\nu=\frac{1}{\log \zeta},\:\:\:\: \kappa=\frac{1}{\sigma_i^2\left(\zeta^{\frac{1}{4}}+\zeta^{-\frac{1}{4}}\right)^2}.$ Since $\alpha_k\geq 0,\forall k$, we can use $\vartheta_k=2\sigma_{i}^{2}(1+\alpha_k)(\xi+1)\log t$ to get the statement of the lemma.
\end{proof}
Using the statement of Lemma \ref{lem:tailProb} in \eqref{eq:twotail},
\begin{align}
\mathbb{P}\left(Q_i^k(t)\geq Q_{i^*}^k(t),N_i^k(t)>\eta_i^k(t)\right)\nonumber \\ 
\leq 
\frac{2}{\log \zeta} \frac{\log \left((1+d_k)t\right)}{t^{(\xi+1)(1+\alpha_k)}}.
\label{inequality2}
\end{align}
Summing the right hand side of \eqref{inequality2} over $t$ we get
\begin{align}
&\sum_{t=1}^T\frac{\log \left((1+d_k)t\right)}{t^{(\xi+1)(1+\alpha_k)}} \leq \log (1+d_k) \nonumber \\ & \! + \frac{\log (1+d_k)(\xi\alpha_k+\xi+\alpha_k)+1}{(\xi\alpha_k+\xi+\alpha_k)^2}.
\label{expand}
\end{align}
Since $\log $ is  concave, substituting \eqref{expand} into \eqref{inequality2} we get
\begin{align}
\sum_{k=1}^K\sum_{t=1}^T\mathbb{P}\left(Q_i^k(t)\geq Q_{i^*}^k(t),N_i^k(t)>\eta_i^k(t)\right)\nonumber\\
\leq \frac{2K}{\log \zeta}\log (1+d_{avg})
\nonumber\\+  \frac{2}{\log \zeta}\sum_{k=1}^K\frac{\log (1+d_k)(\xi\alpha_k+\xi+\alpha_k)+1}{(\xi\alpha_k+\xi+\alpha_k)^2},\label{eq:tail}
\end{align}
which upper bounds the second summation of (\ref{neq:step1regret}). 

Next, we upper bound the first summation term of  (\ref{neq:step1regret}) as follows. Since we restrict to  symmetric graphs where all center agents have the same number and type of neighbors,  $\alpha_k=\alpha, \forall k\leq m$. Then, by \eqref{eq:eta} we have $\eta_i^k(t)=(1+\alpha)\eta_i(t),\forall k\leq m$, and $\eta_i^k(t)=\eta_i(t) ,\forall k>m.$ Let $[x]^+=\max\{x,0\}.$

\begin{lemma} \label{lem:Bound}
Let $G$ be a symmetric multi-star graph with $m$ center agents and $K-m$ peripheral agents. Let $\{\eta_i^k(t)\}_1^T$  be the sequence of nonnegative nondecreasing functions given by \eqref{eq:eta}. Then
\begin{align*}
\sum_{k=1}^K\sum_{t=1}^T\mathbb{P}\left(\mathbb{I}_{\{\varphi_t^k=i\}},N_i^k(t)\leq \eta_i^k(t)\right) 
\leq (K-m)\eta_i(T)\\
+\frac{m}{1+p(m-1)}\left[1-p\frac{K-m}{m}\right]^+(1+\alpha)\eta_i(T).
\end{align*}
\end{lemma}
\begin{proof}
\normalfont
Recall the definitions of $n_i^k(t)$ and $N_i^k(t)$ in \eqref{eq:nandN}.
Since the communication structure is independent of the decision-making process $\forall k$, 
\begin{align}
\mathbb{E}\left(n_i^k(t)\right)+p\sum_{{e(k,j)\in\mathcal{E}}}^K \mathbb{E}\left(n_i^j(t)\right)=\mathbb{E}\left(N_i^k(t)\right).\label{eq:probcond}
\end{align}
Since $N_i^k(t)$ is a nonnegative random variable, $N_i^k(t)\leq \eta_i^k(t)\implies \mathbb{E}\left(N_i^k(t)\right)\leq \eta_i^k(t).$ Thus, from (\ref{eq:probcond}), for all $k$,
\begin{align}
\mathbb{E}&\left(n_i^k(t),N_i^k(t)\leq \eta_i^k(t)\right)\nonumber\\
&+p\sum_{{e(k,j)\in\mathcal{E}}}^K \mathbb{E}\left(n_i^j(t), N_i^k(t)\leq \eta_i^k(t)\right)\leq \eta_i^k(t). \label{eq:expNum}
\end{align}
To upper bound 
$\sum_{k=1}^K\sum_{t=1}^T\mathbb{P}\left(\mathbb{I}_{\{\varphi_t^k=i\}},N_i^k(t)\leq \eta_i^k(t)\right)$ 
we maximize $\sum_{k=1}^K\mathbb{E}(n_i^k(t))$ subject to the constraint given by (\ref{eq:expNum}). This is the linear programming optimization problem: maximize $\sum_{k=1}^K\mathbb{E}\left(n_i^k(t),N_i^k(t)\leq \eta_i^k(t)\right)$ subject to \eqref{eq:expNum} and $\mathbb{E}\left(n_i^k(t),N_i^k(t)\leq \eta_i^k(t)\right)\geq 0$ for all $k$.
For $p=1$ the solution is the sum of $\eta^k_i(t)$ over the maximal independent set of $G$, which for a multi-star graph is the set of peripheral agents $k\geq m+1$. 
Thus, for general $p$ we have
\begin{align*}
\sum_{k=1}^K\sum_{t=1}^T\mathbb{P}\left(\mathbb{I}_{\{\varphi_t^k=i\}},N_i^k(t)\leq \eta_i^k(t)\right)
\leq \sum_{k=m+1}^K\eta_i^k(T)\\
+\sum_{k=1}^m\frac{1}{1+p(m-1)}\left[1-p\frac{K-m}{m}\right]^+\eta_i^k(T),
\end{align*}
and the statement of the lemma follows.
\end{proof}

This concludes upper bounding the first summation of \eqref{neq:step1regret}.

\begin{theorem} \label{thm:mainresut}
Consider a distributed stochastic bandit problem with $N$ options, $K$ agents, and $T$ time steps. Let communication graph $G$ be a symmetric multi-star graph with $m$ center agents and $K-m$ peripheral agents. If all agents sample according to the heterogeneous sampling rules defined in Definition \ref{def:samplerule}, the expected cumulative group regret satisfies
\begin{align*}
&\mathbb{E}\left(R(T)\right) \leq c_1(K,m,\alpha,p)\sum_{i=1}^N\frac{8\sigma_i^2(\xi+1)\log T}{\Delta_i}\\
&
+\frac{2}{\log \zeta}\sum_{i=1}^N\Delta_i \Bigg( K\log (1+d_{avg}) +(K-m)\frac{\xi\log 2 + 1}{\xi^2} \\ 
& \quad \quad\quad \quad \quad  + m  \frac{\log (1+d_{cen})(\xi\alpha+\xi+\alpha) + 1}{(\xi\alpha+\xi+\alpha)^2} \Bigg), \\
& c_1(K,m,\alpha,p) = K-m + \frac{m(1+\alpha)}{1+p(m-1)}\left[1-p\frac{K-m}{m}\right]^+,
\end{align*}
where $\Delta_i$ is the expected reward gap between the options $i^*$ and  $i$, $\sigma_i^2$ is the variance proxy, and $\xi,\zeta>1$.
\end{theorem}

\begin{proof}
\normalfont
Result follows from \eqref{eq:regret}, (\ref{neq:step1regret}), (\ref{eq:tail}) and Lemma \ref{lem:Bound}.
\end{proof}

\begin{remark}\label{rem:homregret}
Recall that under the corresponding homogeneous sampling rules we have $\alpha_k=0, \forall k.$ Thus, we can recover the expected cumulative group regret bound for the homogeneous sampling rules as follows:
\begin{align*}
&\mathbb{E}\left(R(T)\right) \leq c_2(K,m,p)\sum_{i=1}^N\frac{8\sigma_i^2(\xi+1)\log T}{\Delta_i}\\
&
+\frac{2}{\log \zeta}\sum_{i=1}^N\Delta_i \Bigg( K\log (1+d_{avg})  +(K-m)\frac{\xi\log 2 + 1}{\xi^2}
\\ & \quad \quad\quad \quad \quad \quad + m  \frac{\log (1+d_{cen})\xi + 1}{\xi^2} \Bigg), 
\end{align*}
$c_2(K,m,p) = K-m+\frac{m}{1+p(m-1)}\left[1-p\frac{K-m}{m}\right]^+.$
\end{remark}

When the network graph has a large enough ratio of peripheral agents to center agents and a sufficiently high broadcasting probability $p$, i.e. $p(K-m)/m>1$, we have $
\left[1-p\frac{K-m}{m}\right]^+=0$,
which implies $c_1=c_2=K-m.$ And since $\alpha>0$ we have
$$\frac{\log (1+d_{cen})(\xi\alpha+\xi+\alpha) + 1}{(\xi\alpha+\xi+\alpha)^2} 
< \frac{\log (1+d_{cen})\xi + 1}{\xi^2}.$$

Plugging these results into the bounds of Theorem \ref{thm:mainresut} and Remark \ref{rem:homregret}, we see that the heterogeneous sampling rules provide a lower theoretical regret bound than the corresponding homogeneous sampling rules, which implies that the heterogeneous sampling rules provide better group performance than the  homogeneous sampling rules.

\begin{remark}
Our  bounds hold for sub-exponential reward distributions, where 
$X_i$ is a sub-exponential random variable with mean $\mu_i$ and parameters $(\sigma_i^2,b)$ with $b\leq \frac{\sigma_i}{2\sqrt{2(\xi+1)\log T}}.$
\end{remark}

\section{Simulation Results}\label{sec:simulation}
In this section we provide numerical simulations to illustrate results and validate theoretical bounds. For all simulations, we consider 10  options ($N=10$) with Gaussian reward distributions. Expected reward for the optimal option  is $\mu_{i^*}=11$ and for all sub-optimal options $i\neq i^*$ is $\mu_{i}=10$. We let variance associated with all options $i$ be $\sigma_i^2 =1$. Because the expected reward gaps  $\Delta_i=1$, $i\neq i^*$, are equal to the variances $\sigma_i^2 = 1$, it is a challenging problem to distinguish the optimal option from the sub-optimal options. For all simulations, we consider 1000 time steps ($T=1000$) and use 1000 Monte Carlo simulations with $\xi=1.01$. 

We show simulation results for  performance of a group of $K=36$ agents  that communicate over two different symmetric multi-star graphs and use the heterogeneous sampling rules of Definition~\ref{def:samplerule}. We compare to the case when agents use the corresponding homogeneous sampling rules of Definition~\ref{def:homrule}. The first multi-star graph has $m=2$ center agents and $K-m =34$ peripheral agents, with each center agent communicating with 17 peripheral agents and the other center agent. The second multi-star graph has $m=3$ center agents and $K-m=33$ peripheral agents, with each center agent communicating with 11 peripheral agents and the other center agents. In each case,  center agents are interchangeable and  peripheral agents are interchangeable, so the average performance of a center (peripheral) agent is the same as the individual performance of a center (peripheral) agent. 

Figure \ref{fig:netprob} shows how average expected cumulative group regret varies with broadcasting probability $p$ for agents using the heterogeneous rules (dotted) and homogeneous rules (solid). Regret is inversely related to  performance: lower group regret implies higher group performance. Results are plotted on the left for the graph with 2 center agents and on the right for the graph with 3 center agents. When $p=0$ there is no communication at all.  So when $p$ becomes even just a little positive and agents learn about options from their neighbors,  regret falls, i.e.,  group performance rises. 
\begin{figure}[h]
    \centering
    \includegraphics[width=0.45\textwidth]{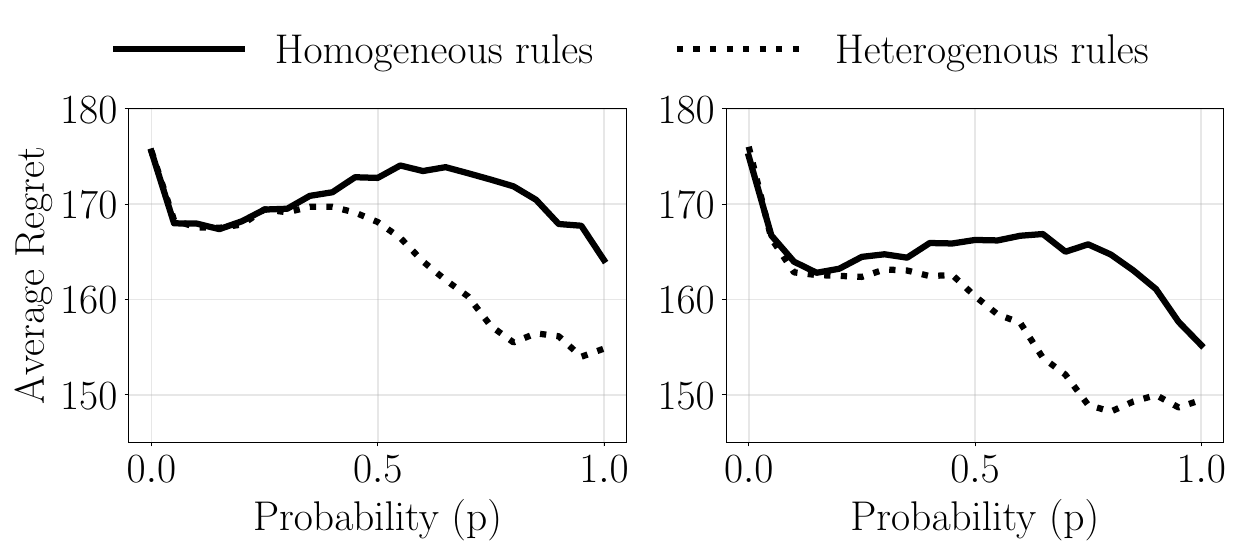}
    \caption{\small{Average expected  cumulative group regret for $K=36$ agents at time $t=1000$ as a function of broadcasting probability $p$ with communication over a symmetric multi-star graph. Left: 2 center and 34 peripheral agents. Right: 3 center and 33 peripheral agents. Dotted lines and solid line shows average regret when agents use  heterogeneous and homogeneous sampling rules, respectively.}
    } 
    \label{fig:netprob}
\end{figure}

In the case of the homogeneous rules, as $p$ increases through intermediate values, center agents do less and less exploring and the usefulness of the information received by peripheral agents decreases. This leads to increased regret for peripheral agents, and the group overall, and thus degraded group performance. When $p$ approaches 1,  center agents receive sufficient  information from their peripheral neighbors such that their improved performance outweighs the degraded performance of peripheral agents.
This leads to a final decrease in group regret and increase in group performance. 

The improvement in performance provided by the heterogeneous rules relative to the homogeneous rules, as predicted by Theorem~\ref{thm:mainresut} and Remark~\ref{rem:homregret}, can be clearly seen in Figure~\ref{fig:netprob} by observing how much lower the dotted regret curve is than the solid regret curve. The growth in regret in the  homogeneous case, as $p$ increases through intermediate values, is reduced in the heterogeneous case. This is because, by design, center agents are biased toward more exploring, which improves the information that peripheral agents receive. The group performance increase that comes, as $p$ increases further, occurs in the heterogeneous case well before $p$ approaches 1.

The influence of irregularity of the graph can be observed in Figure~\ref{fig:netprob} by comparing the left plot (2 center agents and more irregular) to the right plot (3 center agents and less irregular). The results suggest that performance is higher with more center agents, i.e., with greater regularity in the graph. 

Figure \ref{fig:nettime} shows expected cumulative regret as a function of time $t$ for  center (blue), peripheral (pink), and average (black) agents,  when $p=0.8$ and agents use the heterogeneous rules (dotted) and homogeneous rules (solid). Results are plotted on the left for the graph with 2 center agents and on the right for the graph with 3 center agents.  It can be observed that, as predicted for the heterogeneous rules, the peripheral agent performance increases and the center agent performance decreases, such that group performance (as represented by the average agent) improves.  Further, a comparison of left and right plots suggests that group performance improves with more center agents (more regularity).
\begin{figure}[h]
    \centering
    \includegraphics[width=0.45\textwidth]{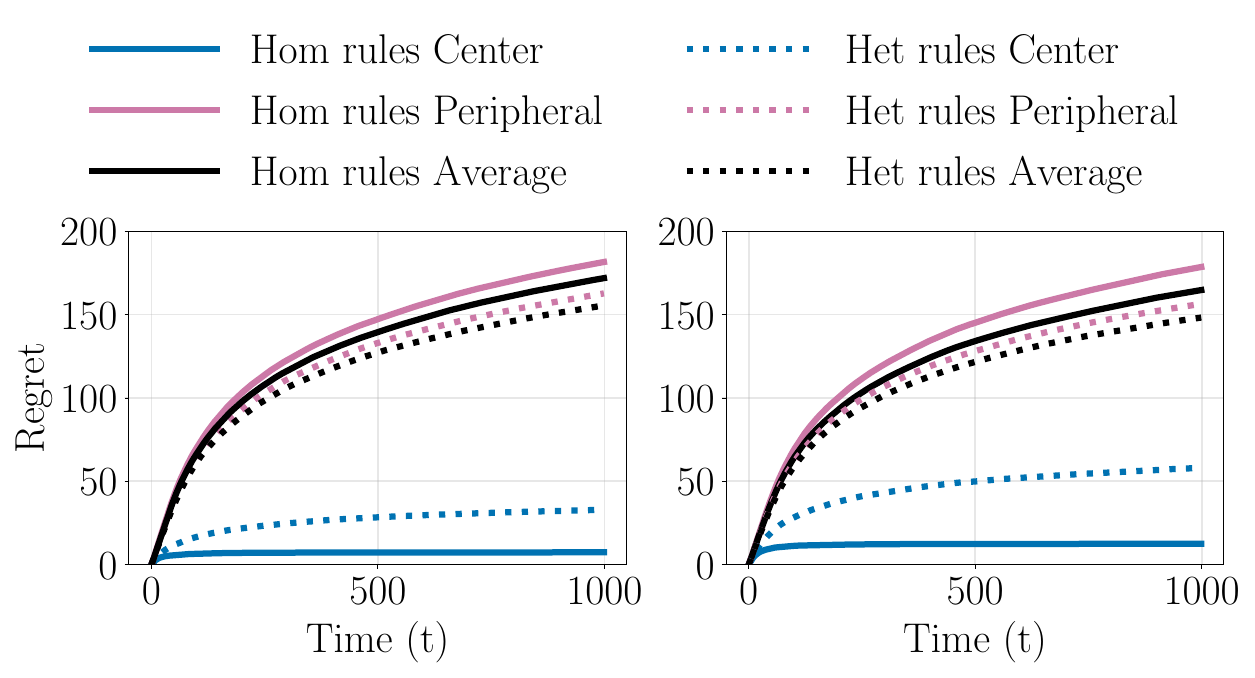}
    \caption{\small{Expected cumulative regret of center agent, peripheral agent, and average agent for $K=36$ agents as a function of time $t$ for $p=0.8$ and the same two symmetric multi-star graphs as in Figure~\ref{fig:netprob}: 2 center agents (left) and 3 center agents (right) where agents use heterogeneous (dotted) and homogeneous (solid) sampling rules.}}
    \label{fig:nettime}
\end{figure}


\section{Conclusions}\label{sec:conl}
We have designed and analyzed new heterogeneous rules for how a group of agents that share information over a network should sample an uncertain environment to maximize group reward. We consider communication networks defined by symmetric multi-star graphs, since these exemplify realistic settings. Using the multi-armed bandit problem as the explore-exploit framework, we show how sampling rules for center agents that favor exploring over exploiting make the information that center agents broadcast to their neighbors more useful, thereby increasing the total reward accumulated by the group.  

Our analysis and design advance  understanding of the role that heterogeneity does and can play in collective decision making.  And our demonstration that heterogeneity can be leveraged to improve the performance of a cooperative multi-agent system suggests that further investigation is warranted.


\bibliographystyle{IEEEtran}
\bibliography{LCC2020}

\end{document}